\documentclass[12pt]{amsart}

\usepackage{epsfig, graphics, psfrag, color, amsmath}
\usepackage[margin=1.15in,]{geometry}

\newcommand{\D}{{\partial}} % Boundary
\theoremstyle{plain}
\newtheorem{theorem}{Theorem}[section]

\newtheorem{cor}[theorem]{Corollary}
\newtheorem{lemma}[theorem]{Lemma}

\newtheorem*{namedtheorem}{\theoremname}
\newcommand{\theoremname}{testing}

\theoremstyle{definition}

\author{Alexander Coward}

\title[Crossing changes and circular Heegaard splittings]
{Crossing changes and circular Heegaard splittings}

\begin{document}

\maketitle

\vspace{-24pt}

\begin{abstract}We use technology from sutured manifold theory and the theory of Heegaard splittings to relate genus reducing crossing changes on knots in $S^3$ to twists on surfaces arising in circular Heegaard splittings for knot complements. In a separate paper, currently in preparation, we  prove that these circular Heegaard splittings may be searched for algorithmically, and together our results imply that an algorithm to detect when two hyperbolic or fibered knots of different genus are related by a crossing change would follow from an algorithm to determine whether two compact oriented surfaces in $S^3$ are related by a single twist. 

\end{abstract}

\section{Introduction}

Let $F$ be a compact surface with boundary, embedded in a $3$-manifold $M$. Let $D$ be a disk embedded in $M$ so that $D \cap F$ is a single arc $\alpha$ properly embedded in $F$ and embedded in the interior of $D$. Let $C = \partial D$. If we perform $\pm 1$ Dehn surgery along $C$ then $F$ is transformed to a new surface $F'$ which  we say is obtained from $F$ by a \emph{twist along $\alpha$}. See Figure \ref{twisting}. 

\begin{figure}[h!]
\centering \label{twisting}
\psfrag{S}[][]{$F$}
\psfrag{a}[][]{$\alpha$}
\psfrag{p}[][]{$\pm 1$}
\psfrag{F}[][]{$F'$}
\includegraphics[width=0.4\textwidth]{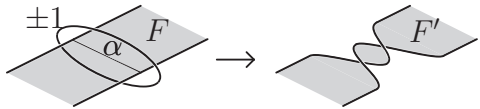}
\caption{Performing a twist along $\alpha$} 
\end{figure} 

Twisting is a very natural operation  one can perform on a surface. Since performing a twist on a surface applies a crossing change to its boundary, understanding twists on surfaces is clearly related to understanding crossing changes on knots. In fact, this statement can be made very strong, by means of the following theorem.

\begin{theorem} \label{maint}
Let $K$ and $K'$ be oriented knots in $S^3$, both either hyperbolic or fibered, with $g(K) > g(K')$. Then there are finite lists of oriented spanning surfaces $\{S_1,\ldots, S_n\}$ for $K$ and $\{S_1',\ldots, S_{n'}'\}$ for $K'$
 such that if $K$ and $K'$ are related by a single crossing change, then some $S_i\in\{S_1,\ldots, S_n\}$ and some $S'_{i'}\in\{S_1',\ldots, S_{n'}'\}$ are related by a single twist, up to an ambient isotopy of $S^3$.
Furthermore, there is an algorithm that will take diagrams
for $K$ and $K'$ as input, and output such finite lists of spanning surfaces.

\end{theorem}

Thoerem \ref{maint} arose from the efforts of the author to construct an algorithm to determine whether a knot in $S^3$ has unknotting number one. The following corollary highlights how Theorem \ref{maint} might be used in this direction.

\begin{cor}
Suppose there is an algorithm to determine whether two compact oriented surfaces in $S^3$, each with a single boundary component, are related by a single twist, up to ambient isotopy of $S^3$. Then there is an algorithm to determine if two given knots in $S^3$ of different genus and both either hyperbolic or fibered are related by a single crossing change. In particular there is an algorithm to determine whether a given knot in $S^3$ has unknotting number one.
\end{cor}

The proof of Theorem  \ref{maint} has many ingredients. To begin with, a classic result of Scharlemann and Thompson \cite{sscross} using deep machinery from sutured manifold theory says that if a knot $K$ in $S^3$ admits a crossing change yielding a knot $K'$ of lower genus then the crossing change may be realized as untwisting a plumbed on Hopf band on some minimal genus Seifert surface $F$ for $K$, as shown in Figure \ref{dehopf}. 
\begin{figure}[h!]
\psfrag{C}[][]{$\pm 1$}
\psfrag{U}[][]{$K'$}
\psfrag{S}[][]{$F$}
\psfrag{F}[][]{$F'$}
\psfrag{K}[][]{$K$}
\centering
\includegraphics[width=0.7\textwidth]{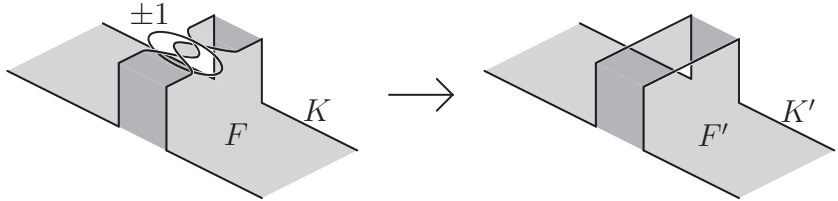}
\caption{Untwisting a plumbed on Hopf band} \label{dehopf}
\end{figure} 

Now, since performing a twist does not change the genus of a surface, the surface $F'$ resulting from untwisting the plumbed on Hopf band on $F$ is certainly not a minimal genus Seifert surface for $K'$. This can readily be seen since $F'$ admits an obvious compression disk. Indeed, if $K'$ is the unknot then $F'$ compresses all the way to a disk. This makes the surface $F'$ considerably more difficult to study than than the incompressible surface $F$. The central idea of this paper is to study highly compressible spanning surfaces for $K'$, such as $F'$, using ideas from another class of highly compressible surface: Heegaard surfaces.

There are essentially two stages in proving Theorem \ref{maint}. The first stage is to show that if two knots of different genus are related by a single crossing change, then this crossing change may be manifested as a twist relating two surfaces arising in certain circular Heegaard splittings for the two knots. The second stage is to algorithmically enumerate all possibilities for these circular Heegaard surfaces. The precise statement we require for the first stage is the following, whose terms are defined in Sections \ref{circpre} and \ref{plumb}. 

\begin{theorem}\label{maintwist}
Let $K$ and $K'$ be two knots in $S^3$. Let $F$ be a genus $n$ Seifert surface for $K$ that may be expressed as the plumbing of two surfaces, $X_1$ and $X_2$, where $X_1$ is a Hopf band, and suppose that untwisting $X_1$ yields a surface whose boundary is ambient isotopic to $K'$. Let $(F,S)$ be a circular Heegaard surface for $K$ with handle number $m$. Then $K'$ admits a circular Heegaard surface $(F',S')$ with thin-genus $n-1$, handle number $m+1$ and the property that $S$ and $S'$ are related by a single twist, up to ambient isotopy. \end{theorem}

The second stage, that of enumerating circular Heegaard surfaces, is inspired by the following theorem of Lackenby \cite{lacksimple}, who built on work of Haken \cite{hakentheorie}, Rubinstein \cite{rubalmost}, Stocking \cite{stocking}, Scharlemann and Thompson \cite{untel}, Casson, Epstein and Penner \cite{ep} and others.

\begin{theorem}\label{lackheegaard}Let $M$ be a compact connected orientable simple 3-manifold with non-empty boundary. Then there is an algorithm to determine the Heegaard genus of $M$.
Moreover, for any given positive integer $n$, there is an algorithm to find all Heegaard
surfaces for $M$ with genus at most $n$ (up to ambient isotopy).
\end{theorem}

An analogous result, whose terms are defined in Section \ref{circpre},  holds for circular Heegaard surfaces for knots in $S^3$.

\begin{theorem}\label{searchy}Let $K$ be a knot in $S^3$ that is either hyperbolic or fibered. Then, up to ambient isotopy of $S^3$ keeping $K$ fixed throughout, there are finitely many circular Heegaard surfaces for $K$ with given thin-genus and given handle number. Furthermore, there is an algorithm to find these surfaces. 
\end{theorem}

The input for the algorithm referred to in Theorem \ref{searchy} is a diagram for $K$ and two integers specifying the thin-genus and handle numbers of the desired circular Heegaard surfaces. The output is a finite list of pairs of surfaces $(F_1,S_1),\ldots,(F_k,S_k)$ each of which forms a circular Heegaard surface for $K$. Note however that there is no guarantee that the surfaces on the list are pairwise non-isotopic. 

Given the semantic similarities, one might expect a proof of Theorem \ref{searchy} to be a simple extension of the ideas used to prove Theorem \ref{lackheegaard}. This, however, is not the case, and there are several technical challenges inherent in proving Theorem \ref{searchy}. Many of these were addressed in previous work of the author \cite{bridgealg} while studying bridge surfaces for hyperbolic knots in $S^3$. Moreover, untelescoping and amalgamating circular Heegaard splittings has several significant, and surprising, differences compared to performing these operations on classical Heegaard splittings. For example, amalgamation (a key tool in the proof of Theorem  \ref{lackheegaard}) is unique for classical generalized Heegaard splittings (see Proposition 3.1 of \cite{lacksimple}), whereas it is not for circular generalized Heegaard splittings. Also, if one untelescopes an unstabilized Heegaard splitting, the resulting generalized Heegaard splitting is also unstabilized. The corresponding statement is false for circular Heegaard splittings. These considerations merit  investigation in their own right, and for this reason we prove Theorem \ref{searchy} in a separate paper \cite{algorithmscirc}.

This paper is arranged as follows: In Section \ref{circpre} we introduce circular Heegaard splittings. In Section  \ref{plumb} we define plumbing and review Scharlemann and Thompson's work \cite{sscross} which relates genus reducing crossing changes on knots in $S^3$ to plumbed on Hopf bands on minimal genus Seifert surfaces. In Section \ref{bigsec} we prove Theorem \ref{maintwist} and in Section \ref{secmain} we prove Theorem \ref{maint}, assuming Theorem \ref{searchy}. The final section of this paper examines an illustrative example: the case there $K$ and $K'$ are both fibered, and provide a proof of Theorem  \ref{maint} in this case that is independent of Theorem \ref{searchy}.

I am grateful to Jessica Banks for some helpful comments she provided after reading a  draft of this paper, and to the Australian Research Council for its support through Discovery grant DP110101104.

\section{Circular Heegaard splittings}\label{circpre}

This paper uses ideas from the theory of Heegaard splittings to understand  genus reducing crossing changes of knots. We therefore begin with some definitions from this theory. 

A \emph{compression body} is a connected orientable $3$-manifold $C$ that is either a handlebody or  obtained from $S \times [0,1]$ by attaching $1$-handles to $S \times \{1\}$, where $S$ is a compact, orientable, possibly disconnected surface with no 2-sphere components. The copy of $S \times \{0\}$ in $C$ is called the \emph{negative boundary} and is denoted $\partial_-C$.  The copy of $\partial S \times I$ in $C$ is called the \emph{vertical boundary}. The negative boundary and vertical boundary are defined to be empty when $C$ is a handlebody. The closure of the rest of the boundary is called the \emph{positive boundary} and is denoted $\partial_+C$. Note that this definition is not quite standard; some authors insist that $S$ is a closed surface, and some do not require $C$ to be connected. 

A \emph{Heegaard splitting} is a decomposition of a compact orientable $3$-manifold $M$ along an orientable properly embedded surface $S$, called the \emph{Heegaard surface} for the splitting, into two compression bodies, $C_1$ and $C_2$, so that $\partial_+C_1 \cap \partial_+C_2 = S$. A \emph{circular Heegaard splitting} is a decomposition of a compact orientable $3$-manifold along two disjoint orientable properly embedded  surfaces $S$ and $F$ into two compression bodies, $C_1$ and $C_2$, so that $\partial_-C_1 \cap \partial_-C_2 = F$ and $\partial_+C_1 \cap \partial_+C_2 = S$. The pair of surfaces $(F,S)$ is called the \emph{circular Heegaard surface} for the circular Heegaard splitting.  We allow the case that $C_1$ and $C_2$ are products. A schematic diagram of a circular Heegaard splitting is shown in Figure \ref{CHS}. 

\begin{figure}[h!]  
\centering
 \psfrag{C}[][]{$C_1$}
 \psfrag{D}[][]{$C_2$}
  \psfrag{S}[][]{$S$}
 \psfrag{F}[][]{$F$}
\includegraphics[width=0.2\textwidth]{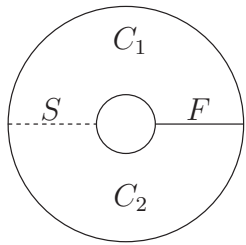}
\caption{A circular Heegaard splitting} \label{CHS}
\end{figure} 

In this paper we will be concerned with circular Heegaard splittings of knot exteriors. To this end let $K$ be an oriented knot in $S^3$ and let $M = \overline{S^3 - \eta(K)}$ be the exterior of $K$. A pair of connected oriented surfaces $(F,S)$ embedded in $S^3$ will be a called a \emph{circular Heegaard surface for $K$} if the following conditions hold: \begin{enumerate}
\item The surfaces $F$ and $S$ are both Seifert surfaces for $K$. That is,  their boundaries both equal $K$ and give the correct orientation. 
\item The surfaces $F$ and $S$ are disjoint in their interiors. 
\item The intersection of $(F,S)$ with $M$ forms a circular Heegaard splitting for $M$. 
\item The surfaces $F$ and $S$ each intersect $\overline{\eta(K)}$ in an annulus that joins $K$ to a longitudinal curve on $\partial(\overline{\eta(K)})$. 
\end{enumerate}

For a circular Heegaard surface $(F,S)$, for either a 3-manifold or a knot, we call $F$ the \emph{thin surface} and we call $S$ the \emph{thick surface}. The number of 1-handles in the compression bodies $C_1$ and $C_2$ is the \emph{handle number} of the splitting. We define the \emph{thick-genus} of the splitting to be the genus of $S$. We define the \emph{thin-genus} of  the splitting to be the multi-set containing the genuses of the components of $F$. We will mainly be interested in the case where $F$ is connected, in which case we abuse notation slightly and say the thin-genus of the splitting is the genus of $F$. When we say that a splitting $(F,S)$ has genus $n$, it is implicit that $F$ is connected. In this case the handle number is simply the difference between the genus of $F$ and the genus of $S$. Note that $S$ is automatically connected. 

In the classical theory, there is a close interplay between Heegaard splittings and Morse theory. The same is true in the circular setting. Consider a circular Heegaard splitting $(F,S)$ for a 3-manifold $M$. Let $S^1$ be identified with the unit circle in the complex plane and define a circular Morse function $f\colon M\rightarrow S^1$ as follows. Put $f^{-1}(1)$ equal to $F$ and $f^{-1}(-1)$ equal to $S$. As $t$ increases from $0$ to $\pi$, we let $f^{-1}(e^{it})$ sweep though parallel surfaces except for exactly one index-$1$ critical point for every  1-handle in $C_1$, and as $t$ increases from $\pi$ to $2\pi$ we let $f^{-1}(e^{it})$ sweep though parallel surfaces except for exactly one index-$2$ critical point for every 1-handle in $C_2$. We call the function $f\colon M\rightarrow S^1$ constructed in this way an \emph{associated circular Morse function} for the circular Heegaard splitting. In the case that $(F,S)$ is a circular Heegaard splitting for a knot in $S^3$, we extend $f$ to the complement of $K$ by radial extension inside $\overline{\eta(K)}$. For more on circular Heegaard splittings for knots and their associated circular Morse functions see \cite{cvtp}.

\section{Plumbing and genus reducing crossing changes}\label{plumb}

The fact, due to Scharlemann and Thompson \cite{sscross}, that a genus reducing crossing change on a knot in $S^3$ can be realized as the untwisting of a Hopf band on a minimal genus Seifert surface is crucial for this paper. We give precise definitions for this statement now. 

Suppose that $S_1$ and $S_2$ are compact orientable surfaces embedded in 3-balls $B_1$ and $B_2$. Suppose that the
intersection of each $S_i$ with $\partial B_i$ is a square $I \times I$ such that
$(I \times I) \cap \partial S_1 = I \times \partial I$ and $(I \times I) \cap
\partial S_2 = \partial I \times I$. Then the surface $S$ in $S^3$ obtained by  \emph{plumbing} $S_1$ and $S_2$
is constructed by gluing the boundaries of $B_1$ and $B_2$ so that
the two copies of $I \times I$ are identified in a way that preserves their product
structures. See Figure \ref{Plumb}.

\begin{figure}[h!]  
 \psfrag{A}[][]{$B_1$}
  \psfrag{B}[][]{$S_1$}
   \psfrag{C}[][]{$B_2$}
    \psfrag{D}[][]{$S_2$}
     \psfrag{S}[][]{$S$}
\centering
\includegraphics[width=0.4\textwidth]{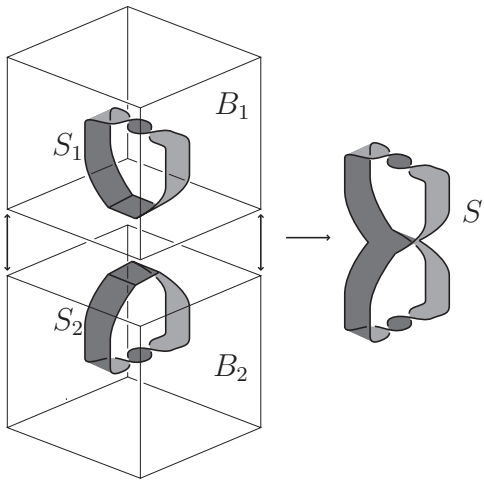}
\caption{The plumbing of two surfaces} \label{Plumb}
\end{figure}

 Suppose that $S_1$ is a Hopf band, which is an unknotted annulus embedded in $B_1$ with a full twist. The \emph{associated crossing disk} $D$ is a disk embedded in the interior of $B_1$ which intersects $S_1$ in a single essential arc in the interior of $D$. The boundary of this disk is the \emph{associated crossing circle}. See Figure \ref{ass}. We refer to preforming $\pm1$ Dehn surgery on the crossing circle so as to remove the full twist of $S_1$ as \emph{untwisting} $S_1$. If we expand $B_1$ a little in all directions, we obtain a 3-ball that intersects $S$ in two arcs properly embedded in $S$. We call this the \emph{encapsulating $3$-ball} of $S_1$. 

\begin{figure}[h!]  
 \psfrag{F}[Bc][Bl]{$\textrm{circle}$}
  \psfrag{A}[Bc][Bl]{$\textrm{crossing}$}
  \psfrag{B}[][]{$\textrm{Hopf band}$}
   \psfrag{C}[][]{$\textrm{crossing disk}$}
    \psfrag{D}[][]{$B_1$}
     \psfrag{E}[][]{$S_1$}

\centering
\includegraphics[width=0.47\textwidth]{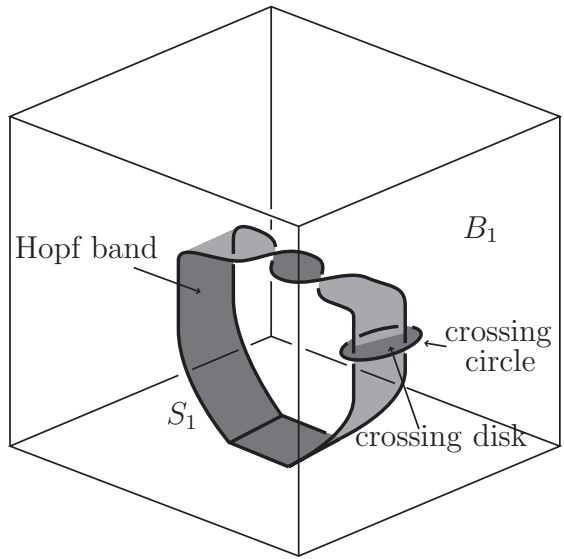}
\caption{A Hopf band} \label{ass}
\end{figure} 

The associated crossing circle for a plumbed on Hopf band is an example of a crossing circle. A \emph{crossing circle} for a knot $K$ in $S^3$ is the boundary of an embedded disk in $S^3$ that interests $K$ in two points of opposite sign. Further, a crossing circle is decorated with a number $x \in \{+1,-1\}$. A \emph{crossing change along $C$} is achieved by performing Dehn surgery along $C$ with slope $x$.

The following result is due to Scharlemann and Thompson and relies on deep results in sutured manifold theory. See Proposition 3.1 of \cite{sscross}.

\begin{theorem} \label{stplumb} Let $C$ be a crossing circle for a non-trivial knot $K$ such that
performing a crossing change along $C$ reduces the genus of $K$. Then $K$ has a minimal genus
Seifert surface which is obtained by plumbing surfaces $S_1$ and $S_2$,
where $S_1$ is a Hopf band. Moreover, there is an ambient isotopy, keeping $K$ fixed
throughout, that takes $C$ to the associated crossing circle for $S_1$, and performing the crossing change along $C$ untwists $S_1$.
\end{theorem}

Theorem \ref{stplumb} is illustrated in Figure \ref{dehopf}. It says that for any genus reducing crossing change on a knot $K$, there exists a minimal genus Seifert surface for $K$ that may be taken inside some 3-ball to look exactly like the left hand surface, $F$, in Figure \ref{dehopf}, and so that performing the crossing change untwists the Hopf band yielding something looking exactly like the right hand surface $F'$. 

\section{Surgery on circular Heegaard splittings}\label{bigsec}

The key construction in this paper relates to how a genus reducing crossing change on a knot with a circular Heegaard splitting gives rise to a new circular Heegaard splitting for the knot after the crossing change. We first illustrate our methods for the case when the initial knot is fibered. 

\begin{theorem}\label{fiberswitch}
Let $K \subseteq S^3$ be a fibered knot with fiber $S$. Suppose that $S$ is the plumbing of two surfaces $X_1$ and $X_2$, where $X_1$ is a Hopf band. Let $K'$ be the knot obtained by performing a crossing change along the crossing circle associated to $X_1$. Then $K'$ admits a handle number one circular Heegaard surface $(F',S')$ with the property that $S'$ is ambient isotopic in $S^3$ to the surface obtained from $S$ by untwisting $X_1$.
\end{theorem}

\begin{proof}
Let $f\colon S^3-K \rightarrow S^1$ be the fibration of the knot complement with $f^{-1}(1) = S$. (Throughout we regard $S^1$ as the  unit circle in the complex plane.) Let $B$ be the encapsulating $3$-ball for $X_1$. We seek to understand the restriction of $f$ to $B$.

\begin{figure}[h!]  
 \psfrag{A}[cc][Bl]{$D_1$}
 \psfrag{B}[cc][Bl]{$B$}
 \psfrag{F}[cc][Bl]{$S=S_0$}
 \psfrag{K}[cc][Bl]{$K$}
\centering
\includegraphics[width=0.4\textwidth]{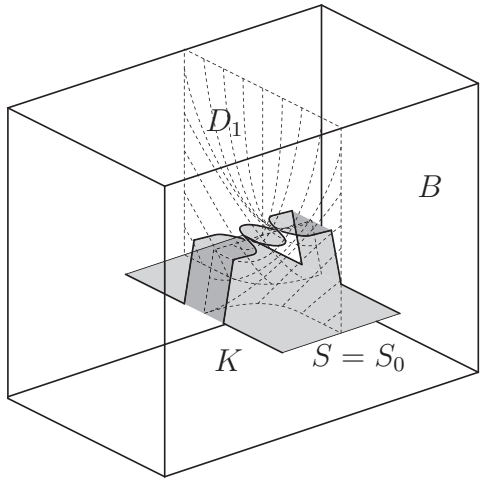}
\caption{The boundary compression disk $D_1$} \label{push}
\end{figure} 

Let $S_t$ be the surface $\overline{f^{-1}(e^{it})}$ for $t\in[0,2\pi)$, so that $S=S_0$. Consider $S \cap B$ as an embedded, but not properly embedded, surface in $B$ and observe that $S\cap B$ admits two boundary compression disks, $D_1$ and $D_2$, in $B$, described as follows. The first, $D_1$, is illustrated in Figure \ref{push}. It has boundary that consists of two arcs both joining points in distinct components of  the interior of $S \cap \partial B$. One arc lies on $\partial B$, missing $S$ in its interior, running over the top of $B$ in Figure \ref{push}. The other arc lies on the interior of $S$, running once around the Hopf band $X_1$. See Figure \ref{push}. The disk $D_1$ is disjoint from $K$, and disjoint from $S$ in its interior. Note how the presence of a full twist in the Hopf band allows for the construction of $D_1$. 

\begin{figure}[h!]  
 \psfrag{A}[cc][Bl]{$D_1$}
 \psfrag{B}[cc][Bl]{$B$}
 \psfrag{F}[cc][Bl]{$S_0$}
 \psfrag{K}[cc][Bl]{$K$}
 \psfrag{S}[cc][Bl]{$D_2$}
\centering
\includegraphics[width=0.4\textwidth]{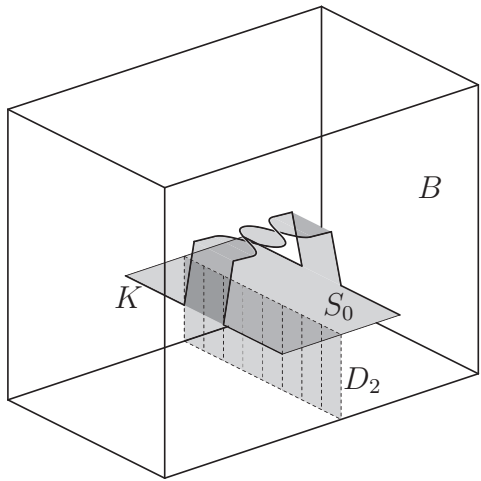}
\caption{The boundary compression disk $D_2$} \label{push2}
\end{figure} 

The other disk, $D_2$, is easier to describe. It also has boundary consisting of two arcs, one running along $\partial B$ and the other running along $S$, as before. However now the disk emanates from the other side of $S$, and the  arc on $S$ does not run around $X_1$. The disk $D_2$ is illustrated in Figure \ref{push2}.

Note that if we push $S$ across $D_1$, we only isotope $S=S_0$ rel $K$ in $S^3$. Therefore, after an isotopy of $f$,  we may take the surface obtained by pushing $S_0$ across $D_1$, and then pushing the whole surface a little more to make it disjoint from $S_0$, to be $S_{\frac{\pi}{2}}$. Similarly we may push $S_0$ in the other direction rel $K$ across $D_2$, and then a little more, and take the resulting surface to be $S_{\frac{3\pi}{2}}$.

The surfaces $S_{\frac{\pi}{2}} \cap B$ and $S_{\frac{3\pi}{2}} \cap B$ are both isotopic rel $K$ in $B$ to the surface $S_{\pi}$ shown in Figure \ref{push3}.

\begin{figure}[h!]  
 \psfrag{A}[cc][Bl]{$D_1$}
 \psfrag{B}[cc][Bl]{$B$}
 \psfrag{F}[cc][Bl]{$S_\pi$}
 \psfrag{K}[cc][Bl]{$K$}
 \psfrag{S}[cc][Bl]{$D_2$}
  \psfrag{P}[cc][Bl]{$\alpha$}
\centering
\includegraphics[width=0.4\textwidth]{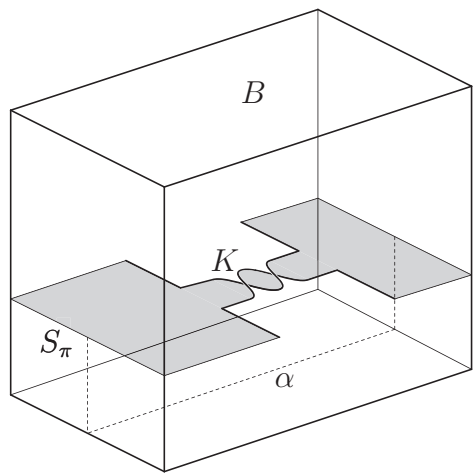}
\caption{The surface $S_\pi$} \label{push3}
\end{figure}

To recap, we may arrange that $f$ interacts with $B$ as follows. The fiber surface 
$S_t$ starts as $S_0$ and, as $t$ increases from $0$, $S_t$ begins by moving in an upward direction. As $t$ increases past some time $t_1$ a little less than $t=\frac{\pi}{2}$ the two arcs of $S_t \cap \D B$ pinch together and form arcs joining the points of $K \cap \D B$ differently. This corresponds to moving a saddle of $S_t$ through $\D B$, although in $S^3$ the surfaces $S_t$ are only sweeping out via an isotopy rel $K$. As $t$ increases to $\pi$, the two `flaps' on either side of $S_t \cap B$ sweep downwards, looking like the surface in Figure \ref{push3} at time $t=\pi$. 

As $t$ increases past $\pi$, the two `flaps' on either side of the surface shown in Figure \ref{push3} continue to move downwards until some time $t_2$ just after $t=\frac{3\pi}{2}$ when the two flaps pinch together. At this time, inside $B$ the surface $S_t$ is changed by the addition of a band that runs along the arc $\alpha$, shown in Figure \ref{push3}. We have now arrived back at a surface that is isotopic rel $K$ in $B$ to $S_0$, and the sweepout continues through parallel surfaces until we arrive back at the surface $S_0$ we started with. 

There are two times $t$ when $S_t \cap B$ changes by something other than an isotopy rel $K$ in $B$. At the first time, $t_1$, a little before  time $\frac{\pi}{2}$, we see $S_t \cap B$ changing by a boundary compression along $D_1$. At the same time, we see $S_t \cap \overline{(S^3 - B)}$ changing by the addition of a band. At time $t_2$, a little after time $\frac{3\pi}{2}$, the situation is reversed: a band is added inside $B$, but on the outside of $B$ the surface $S_t \cap \overline{(S^3 - B)}$ changes by a boundary compression. 

Our strategy to prove Theorem \ref{fiberswitch} is to change $f$ by an explicit replacement inside $B$. To this end, first replace $K$ with the knot $K'$ which agrees with $K$ outside $B$, and inside $B$ consists of two arcs joining the four points of $K \cap \partial B$ as shown in Figure \ref{aftpush1}. This corresponds to performing a crossing change to $K$ along the crossing circle associated to $X_1$, as stipulated in the statement of Theorem \ref{fiberswitch}. Define $f'\colon S^3 - K' \rightarrow S^1$ by setting $f'(x)=f(x)$ for $x\in\overline{(S^3-B)}-K$. It remains to specify $f'(x)$ for $x$ in the interior of $B$. This amounts to specifying, for each $t\in[0,2\pi)$, the intersection of  $\overline{f'^{-1}(e^{it})}$ with $B$. We denote the surface $\overline{f'^{-1}(e^{it})}$ by $S_t'$.

Start by letting $S_0'$ intersect $B$ in the fashion illustrated in Figure \ref{aftpush1}. Note that $S_0' \cap \D B = S_0 \cap \D B$. Also note that $K' = \D S_0'$.

\begin{figure}[h!]  
 \psfrag{A}[cc][Bl]{$\gamma$}
 \psfrag{B}[cc][Bl]{$B$}
 \psfrag{F}[cc][Bl]{$S_\pi$}
 \psfrag{K}[cc][Bl]{$K'$}
 \psfrag{S}[cc][Bl]{$S_0'$}
  \psfrag{P}[cc][Bl]{$\alpha$}
\centering
\includegraphics[width=0.4\textwidth]{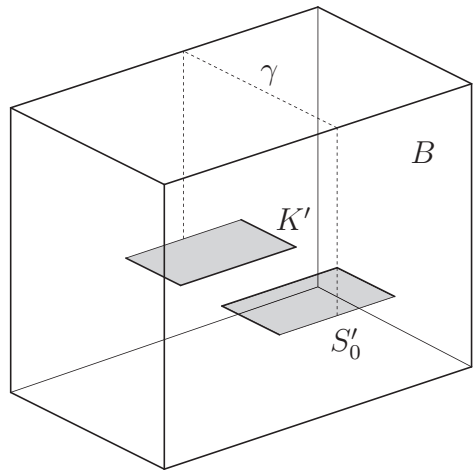}
\caption{The surface $S_0'$} \label{aftpush1}
\end{figure} 

As $t$ increases from $0$, set $S_t'$ to be the family of surfaces that sweep upwards rel $K'$ maintaining $S_t' \cap \D B = S_t \cap \D B$. As $t$ approaches $t_1$, the two arcs of $S_t' \cap \D B$ pinch together in a similar fashion to the two arcs of $S_t \cap \D B$ at this time. However there is now no boundary compression disk to push across. Instead we let $S_t' \cap B$ change by the addition of a band along the arc $\gamma$, shown in  Figure \ref{aftpush1}.

The resulting surface $S_t'$ is isotopic rel $K'$ in $B$ to the surface $S_\pi'$ illustrated in Figure \ref{aftpush2}.

\begin{figure}[h!]  
 \psfrag{A}[cc][Bl]{$\gamma$}
 \psfrag{B}[cc][Bl]{$B$}
 \psfrag{F}[cc][Bl]{$S_\pi$'}
 \psfrag{K}[cc][Bl]{$K'$}
 \psfrag{S}[cc][Bl]{$S_\pi'$}
  \psfrag{P}[cc][Bl]{$\alpha$}
\centering
\includegraphics[width=0.4\textwidth]{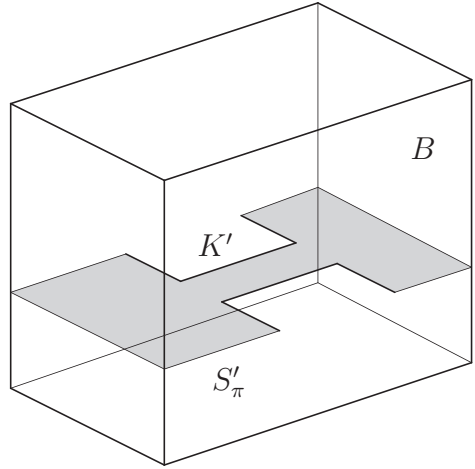}
\caption{The surface $S_\pi'$} \label{aftpush2}
\end{figure} 

As $t$ increases past $t_1$, the two `flaps' of $S_t'$ on the left and right of Figure \ref{aftpush2} sweep downwards, looking like the surface in Figure \ref{aftpush2} at time $t=\pi$. Eventually, as $t$ approaches time $t_2$ the two flaps pinch together once more, but this time we do not add a band to $S_\pi'$. Instead, we facilitate the pinching together of the two flaps of $S_t'$ by boundary compressing along  the disk $D_2$, shown in Figure \ref{d2}.

\begin{figure}[h!]  
 \psfrag{A}[cc][Bl]{$\gamma$}
 \psfrag{B}[cc][Bl]{$B$}
 \psfrag{F}[cc][Bl]{$S_\pi$'}
 \psfrag{K}[cc][Bl]{$K'$}
 \psfrag{S}[cc][Bl]{$D_2$}
  \psfrag{P}[cc][Bl]{$\alpha$}
\centering
\includegraphics[width=0.4\textwidth]{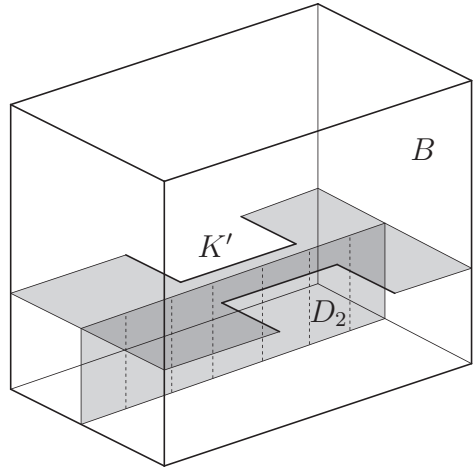}
\caption{The boundary compression disk $D_2$} \label{d2}
\end{figure} 

After boundary compressing along $D_2$ at time $t_2$, the surface $S_t' \cap B$ is now isotopic rel $K'$ in $B$ to the initial surface $S_0'$. We complete the sweepout by allowing $S'_t$ to sweep back up through parallel surfaces to $S_0'$ where it started. 

It is worth emphasizing that we have made sure  $S_t' \cap \D B = S_t \cap \D B$ throughout, even at times $t_1$ and $t_2$. Thus we have constructed a family of surfaces $S_t'$ which specify $f'$. Next we consider the global appearance of the surfaces $S_t'$. To begin with, $S_t'$ starts as $S_0'$ at time $t=0$ and as $t$ increases from $0$ it sweeps upwards rel $K'$ through $S^3$. Then at time $t_1$ a band is added both on the inside of $B$ and the outside of $B$. The overall effect of this is that as $t$ increases past $t_1$, the surfaces $S_t'$ change via the addition of a small tube that runs along the arc $\gamma$  shown in Figure \ref{aftpush1}. Another way of saying this is that $f'$ has an index-1 critical point when $t=t_1$. 

The surfaces $S_t'$ continue to sweep out through $S^3$ via an isotopy rel $K'$ until time $t=t_2$. At this time the surface $S_t' \cap B$ changes via a boundary compression in $B$ along $D_2$, shown in Figure \ref{d2}, and the surface $S_t' \cap \overline{(S^3 - B)}$ also changes via a boundary compression in $\overline{S^3 - B}$. The overall effect of these two boundary compressions is that as $t$ increases from just below $t_2$ to just above $t_2$, the surface $S_t'$ changes by a compression. In other words, $f'$ has an index-2 critical point at time $t=t_2$. 

Put  $F' = S_0'$ and $S' = S'_\pi$. By the comments in the previous two paragraphs, $(F',S')$ forms a handle number one circular Heegaard splitting for $K'$. Also, $S'$ is isotopic rel $K'$ to the surface $F'$ with a tube added along the arc $\beta$, shown in Figure \ref{beta}. This in turn is ambient isotopic to the surface obtained from $S$ by untwisting $X_1$.\begin{figure}[h!]  
 \psfrag{A}[cc][Bl]{$\beta$}
 \psfrag{B}[cc][Bl]{$B$}
 \psfrag{F}[cc][Bl]{$S_\pi$}
 \psfrag{K}[cc][Bl]{$K'$}
 \psfrag{S}[cc][Bl]{$S_0'$}
  \psfrag{P}[cc][Bl]{$\alpha$}
\centering
\includegraphics[width=0.4\textwidth]{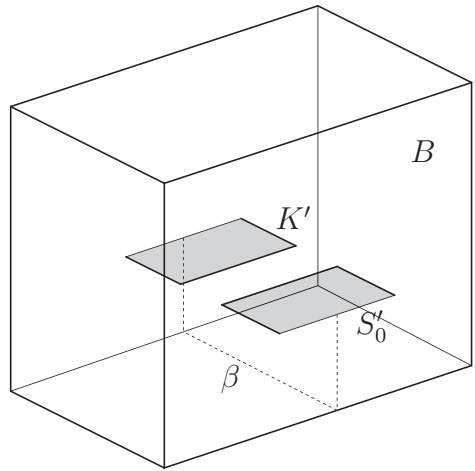}
\caption{The arc $\beta$} \label{beta}
\end{figure}
 \end{proof}

A schematic illustrating the construction in the proof of Theorem \ref{fiberswitch} is shown in Figure \ref{schematicA}.

\begin{figure}[h!]  
 \psfrag{A}[cc][Bl]{$\beta$}
 \psfrag{B}[cc][Bl]{$B$}
 \psfrag{F}[Bc][Bl]{$S=S_0$}
 \psfrag{K}[cc][Bl]{$K$}
 \psfrag{H}[Bc][Bl]{$F'=S_0'$}
  \psfrag{J}[cc][Bl]{$K'$}
    \psfrag{W}[cc][Bl]{$S_\pi$}
 \psfrag{S}[Bc][Bl]{$S'=S_\pi'$}
  \psfrag{V}[Bl][Bl]{$S'_{t_1+\varepsilon}$}
  \psfrag{Q}[Bl][Bl]{$S'_{t_2-\varepsilon}$}
  \psfrag{P}[cc][Bl]{$\alpha$}
\centering
\includegraphics[width=\textwidth]{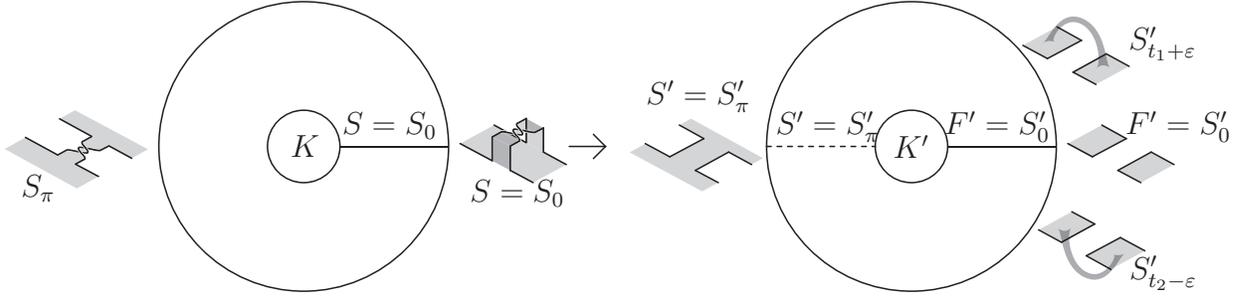}
\caption{Schematic showing the surfaces in Theorem \ref{fiberswitch}} \label{schematicA}
\end{figure} 

We may generalize Theorem \ref{fiberswitch} to the case where $K$ is any knot in $S^3$ with the following theorem, which immediately implies Theorem \ref{maintwist}. 

\begin{theorem}\label{generalswap}
Let $K \subseteq S^3$ be an oriented knot. Suppose that $F$ and $S$ are disjoint Seifert surfaces for $K$ forming a circular Heegaard splitting $(F,S)$ for $K$. Suppose that $F$ is the plumbing of two surfaces $X_1$ and $X_2$, where $X_1$ is a Hopf band. 

Let $F_{\textrm{untwist}}$ be the surface obtained from $F$ by untwisting $X_1$, and let $K' = \partial F_{\textrm{untwist}}$. Then ${K'}$ admits a circular Heegaard splitting $(F',S')$ with the following properties:\begin{enumerate}
\item The surfaces $S$ and $S'$ are related by a  twist, up to ambient isotopy of $S^3$. 
\item The surface $F'$ may be obtained from $F_{\textrm{untwist}}$ by performing a single compression, up to ambient isotopy of $S^3$. 
\item The handle number of $(F',S')$ is one more than that of $(F,S)$. 
\end{enumerate}
\end{theorem}

\begin{proof}

Let $B$ be the encapsulating 3-ball for $X_1$, and let $D_1$ and $D_2$ be boundary compression disks identical to the disks $D_1$ and $D_2$ in $B$ in the proof of Theorem \ref{fiberswitch}. Let $f\colon S^3-K\rightarrow S^1$ be a circular Morse function associated with the circular Heegaard splitting $(F,S)$. 

We claim that after an isotopy rel $K$ of $f$ we may take $f$ to restrict to $B$ in exactly the same way as $f$ restricted to $B$ in the proof of Theorem \ref{fiberswitch}. This is achieved as follows. Let the two sides of $F$ be the $+$ side and the $-$ side, with $D_1$ emanating from the $+$ side and $D_2$ emanating from the $-$ side.  The circular Heegaard surface $(F,S)$ decomposes $M = \overline{S^3 - \eta(K)}$ into two compression bodies. We call these $C_1$ and $C_2$, with $D_1$ in $C_1$ and $D_2$ in $C_2$.

Let $\{\alpha_1,\ldots,\alpha_k\}$ and $\{\beta_1,\ldots,\beta_k\}$ be a choice of cores for the 1-handles in $C_1$ and $C_2$ respectively. More precisely, $\{\alpha_1,\ldots,\alpha_k\}$ and $\{\beta_1,\ldots,\beta_k\}$ are disjoint collections of embedded arcs satisfying the following properties:\begin{enumerate}
\item The arcs $\{\alpha_1,\ldots,\alpha_k\}$ (resp. $\{\beta_1,\ldots,\beta_k\}$) have endpoints on $F$, they emanate from the $+$ (resp. $-$) side, and are disjoint from $F$ in their interiors.
\item The surface obtained by adding tubes to $\overline{F - \eta(K)}$ along all the arcs $\{\alpha_1,\ldots,\alpha_k\}$ (resp. $\{\beta_1,\ldots,\beta_k\}$) cobounds a product region with $\overline{S - \eta(K)}$.
\end{enumerate}

To isotope $f$ rel $K$ so that the restriction of $f$ to $B$ is as in the proof of Theorem \ref{fiberswitch} it will be enough to choose the arcs $\{\alpha_1,\ldots,\alpha_k\}$ and $\{\beta_1,\ldots,\beta_k\}$ to be disjoint from $B$. For this we  use the following straightforward generalization of Haken's Lemma.

\begin{lemma}\label{cglemma}
Let $S$ be a Heegaard surface for an irreducible 3-manifold $M$, decomposing $M$ into compression bodies $C_1$ and $C_2$. Let $D$ be a properly embedded disk in $M$ whose boundary consists of a properly embedded arc on $\D_-C_1$, a properly embedded arc on $\D_-C_2$ and two vertical arcs on the union of the vertical boundaries of $C_1$ and $C_2$. Then $S$ is ambient isotopic rel boundary to a surface $S'$ with the following properties: \begin{enumerate}
\item The surface $S'$ intersects $D$ in a single arc. 
\item There are collections of compression disks on each side of $S'$ which are disjoint from $D$ and along which $S'$ compresses to surfaces parallel to $\D_-C_1$ or $\D_-C_2$ respectively. 
\end{enumerate}
\end{lemma}

The proof of Lemma \ref{cglemma} is a straightforward application of the methods used to prove Lemma 1.1 in \cite{cassgord}. See also \cite{bonotal, hakensome,jacolec}.

We will apply Lemma \ref{cglemma} in the manifold obtained by cutting $\overline{S^3 - \eta(K)}$ along $F$. We take the disk $D$ to be the clean alternating product disk associated with $X_1$, illustrated in Figure \ref{push}. See Section 2 of \cite{UGO} for the definition of a clean alternating product disk. Lemma \ref{cglemma} tells us that we may take the arcs $\{\alpha_1,\ldots,\alpha_k\}$ and $\{\beta_1,\ldots,\beta_k\}$ to be disjoint from $D$. This is achieved by taking them to be dual to the collections of compression disks in conclusion (2).
\begin{figure}[h!]  
 \psfrag{D}[cc][Bl]{$D$}
 \psfrag{B}[cc][Bl]{$B$}
 \psfrag{F}[cc][Bl]{$F$}
 \psfrag{K}[cc][Bl]{$K$}
\centering
\includegraphics[width=0.4\textwidth]{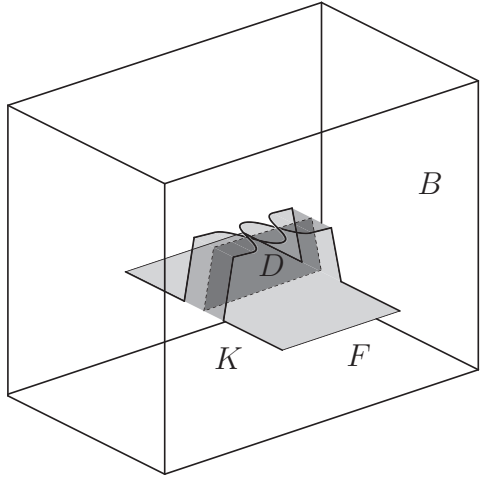}
\caption{A clean alternating product disk} \label{pushw}
\end{figure} 

Once the arcs $\{\alpha_1,\ldots,\alpha_k\}$ and $\{\beta_1,\ldots,\beta_k\}$ are disjoint from $D$ we may push them out of $B$.  The function $f$ then has fibers $S_t:=\overline{f^{-1}(e^{it})}$ that evolve through $S^3$ as follows. We start at time $t=0$ with  $S_0 = F$. As $t$ increases, we begin by allowing the surfaces $S_t$ to evolve by pushing across a disk $D_1$ as shown in Figure \ref{push}. Let $t_1$ be the time when $S_t$ intersects $\D B$ non-transversely in single non-degenerate saddle. Shortly after time $t_1$ we add tubes along the arcs $\{\alpha_1,\ldots,\alpha_k\}$ meaning that we give $f$ a single index-1 critical point for each tube. Continue to sweep through parallel surfaces until time $t=\pi$ when the surface $S_\pi$ intersects $B$ as shown in Figure \ref{push3}. This completes the description of $S_t=\overline{f^{-1}(e^{it})}$ for $t\in[0,\pi]$. 

For $t\in[\pi,2\pi)$, it is easier to visualize what happens when $t$ decreases from $2\pi$. To begin with, as $t$ decreases from $2\pi$, $S_t$ is pushed across a disk $D_2$ as illustrated in Figure \ref{push2}. Let $t_2$ be the  time when $S_t$ intersects $\D B$ non-transversely in a single non-degenerate saddle. Shortly after time $t=t_2$ (that is, at times $t$ slightly smaller than $t_2$) we add tubes  along the arcs $\{\beta_1,\ldots,\beta_k\}$. This means that $f$ has a single index-2 critical point for each arc $\{\beta_1,\ldots,\beta_k\}$. The surface that results is now parallel to $S_\pi$.  The sweepout may continue through parallel surfaces until we reach time $t=\pi$ and a surface $S_\pi$ that agrees with that obtained when we specified $S_t$ for $t\in[0,\pi]$. We have now specified the function $f$ so that the restriction of $f$ to $B$ is exactly as in the proof of Theorem \ref{fiberswitch}.

We may explicitly replace $f$ with a function $f'\colon S^3 - K'\rightarrow{S^1}$ in the exact same fashion as in the proof of Theorem \ref{fiberswitch}. That is, we take $K'$ to agree with $K$ outside $B$ and inside $B$ it as shown in Figure \ref{beta}. Moreover we take $f'=f$ outside $B$ and inside $B$ we define $f'$ to be exactly as $f'$ is defined inside $B$ in the proof of Theorem \ref{fiberswitch}.

For $t\in[0,2\pi)$ put $S'_t=\overline{f'^{-1}(e^{it})}$. Note that $F_{\textrm{untwist}}$ is ambient isotopic to $S'_{t_2-\varepsilon}$. Set $F'=S'_0$ and $S' = S'_\pi$.  Then $(F',S')$ forms a circular Heegaard splitting for  $K'$ satisfying conclusions (1), (2) and (3) in the statement of Theorem \ref{generalswap}. This completes the proof  of Theorem \ref{generalswap}. \end{proof}

A schematic of the surfaces constructed in the proof of theorem \ref{generalswap} is shown in Figure \ref{schematicB}.

\begin{figure}[h!]  
 \psfrag{A}[cc][Bl]{$\beta$}
 \psfrag{B}[cc][Bl]{$B$}
 \psfrag{F}[Bc][Bl]{$F=S_0$}
 \psfrag{K}[cc][Bl]{$K$}
 \psfrag{H}[Bc][Bl]{$F'=S_0'$}
  \psfrag{J}[cc][Bl]{$K'$}
    \psfrag{W}[cr][Br]{$S=S_\pi$}
 \psfrag{S}[Bc][Bl]{$S'=S_\pi'$}
  \psfrag{V}[Bl][Bl]{$S'_{t_1+\varepsilon}$}
  \psfrag{Q}[Bl][Bl]{$S'_{t_2-\varepsilon}$}
  \psfrag{P}[Bl][Bl]{$=F_{\textrm{untwist}}$}

    \psfrag{L}[Bc][Bl]{$S=S_\pi$}
\centering
\includegraphics[width=\textwidth]{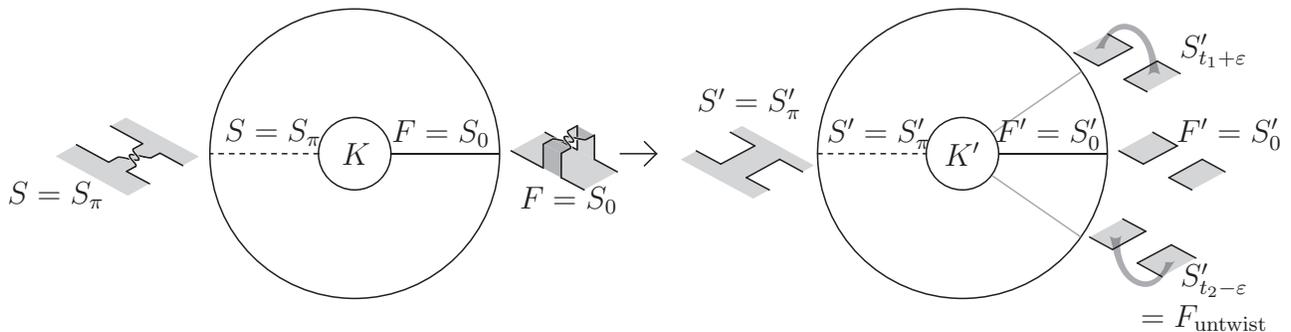}
\caption{Schematic showing the surfaces in Theorem \ref{generalswap}} \label{schematicB}
\end{figure} 

\section{Proof of Theorem \ref{maint}}\label{secmain}

In this section we prove the main theorem of this paper, Theorem \ref{maint}, assuming Theorem \ref{searchy} \cite{algorithmscirc}.

\begin{proof}[Proof of Theorem \ref{maint}] Let $K$ and $K'$ be oriented knots in $S^3$ with the genus of $K$ greater than the genus of $K'$. Suppose that $K$ and $K'$ are both either hyperbolic or fibered. We will describe how to construct the lists of oriented spanning surfaces $\{S_1,\ldots,S_n\}$ for $K$ and $\{S_1',\ldots,S_{n'}'\}$ for $K'$ as stipulated in the statement of Theorem \ref{maint}.

\vspace{8pt}
\textbf{Step 1:} First determine whether $K$ is fibered. This may be achieved using standard results from normal surface theory \cite{mat}. See Algorithm 6 of \cite{tillcoop}. If $K$ is fibered let $\{S_1,\ldots,S_n\}$ be the set containing just a fiber of $K$. This may be found using the algorithm described in the proof of Theorem 4.1.10 of \cite{mat}. We need to keep track of an integer $m$ in the algorithm. In the case that $K$ is fibered put $m=0$.

If $K$ is not fibered (implying that it is hyperbolic), find a list $\{F_1,\ldots,F_n\}$ containing every minimal genus Seifert surface for $K$, up to ambient isotopy of $S^3$ keeping $K$ fixed throughout. This may be achieved using normal surface theory in a suitable ideal triangulation for the knot exterior. See \cite{lacksimple,mat}.

For each minimal genus Seifert surface $F_i$, find a circular Heegaard splitting $(F_i,S_i)$. This may be achieved by cutting along $F_i$ and using a similar procedure to that described in Theorem 3.1.10 of \cite{lectureheeg}. This specifies the list of surfaces $\{S_1,\ldots,S_n\}$. In this case let $m$ be the maximal handle number of the circular Heegaard splittings $(F_i,S_i)$.

\vspace{8pt}
\textbf{Step 2:} Use Theorem \ref{searchy} to find a list $\{(F_1,S_1),\ldots,(F_{n'},S_{n'})\}$ containing every circular Heegaard splitting for $K'$ with thin-genus $g(K)-1$ and handle number at most $m+1$, up to ambient isotopy keeping $K'$ fixed. This specifies the list $\{S_1',\ldots,S_{n'}'\}$.

\vspace{8pt}
It remains to prove that the lists $\{S_1,\ldots,S_n\}$ and $\{S_1',\ldots,S_{n'}'\}$ have the property that if $K$ and $K'$ are related by a single crossing change then some $S_i$ and some $S'_{i'}$ are related by a twist. So suppose that $K$ and $K'$ a related by a single crossing change. Then by Theorem \ref{stplumb}, $K$ has a minimal genus Seifert surface $F$ which may be expressed as the plumbing of two surfaces, $X_1$ and $X_2$, where $X_1$ is a Hopf band, and which has the property that untwisting $X_1$ yields a surface whose boundary is ambient isotopic to $K'$.

If $K$ is fibered, then $F$ must appear on the list $\{S_1,\ldots,S_n\}$. By Theorem \ref{fiberswitch} the surface obtained by untwisting $X_1$ must appear on the list $\{S_1',\ldots,S_{n'}'\}$. This completes the proof of Theorem \ref{maint} in the case that $K$ is fibered, assuming Theorem \ref{searchy}.

If $K$ is not fibered, let $S$ be the surface in the list $\{S_1,\ldots,S_n\}$ that together with $F$ forms a circular Heegaard splitting for $K$. By Theorem \ref{generalswap}, $S$ admits a twist to yield a surface $S'$ that arises as the thick surface in a circular Heegaard splitting for $K'$ with thin-genus $g(K)-1$ and handle number at most $m+1$. Hence $S'$ appears on the list $\{S_1',\ldots,S_{n'}'\}$ after an ambient isotopy of $S^3$. This completes the proof of Theorem \ref{maint} in the case that $K$ is not fibered, assuming Theorem \ref{searchy}.\end{proof}

\section{An illustrative  example: When $K$ and $K'$ are both fibered}

In the case where $K$ and $K'$ are both fibered, the proof of Theorem \ref{maint} relies on a special case of Theorem \ref{searchy} that illustrates some of the methods we explore in more detail in \cite{algorithmscirc}. Remarkably, this case was studied by Kobayashi \cite{koblink} as far back as 1989, although he used no language from the theory of Heegaard splittings, much of which had yet to be developed.  

\begin{theorem}\label{example}Let $K$ be a fibered knot in $S^3$. Then $K$ has a unique handle number one circular Heegaard splitting of every thin-genus at least $g(K)$, up to ambient isotopy of $S^3$ keeping $K$ fixed throughout. Furthermore there is an algorithm to construct this circular Heegaard surface for any given thin-genus. 
\end{theorem}

\begin{proof}
Let $(F,S)$ be a handle number one circular Heegaard splitting for $K$, decomposing $M=\overline{S^3-\eta(K)}$ into two compression bodies $C_1$ and $C_2$. Since $(F,S)$ is handle number one, $S$ admits a unique (up to ambient isotopy) compression disk $D_1$ (resp. $D_2$) in $C_1$ (resp. $C_2$). 

There are two options regarding $F$: either it is incompressible or it is not. The former case holds precisely when $F$ has thin-genus $g(K)$ and the later case holds precisely when $F$ has thin-genus strictly greater than $g(K)$.

If $F$ is incompressible then it must be a fiber. Thus, after cutting along $F$, the surface $S$ forms a Heegaard splitting for $F \times I$. By the main theorem of \cite{heegstandard}, this must be a once stabilized copy of $F \times \{\frac{1}{2}\}$. (Note that the setting in \cite{heegstandard} was $F \times I$ where $F$ is a closed surface, but the methods presented there readily extend to the case where $F$ is a surface with boundary.) This completes the proof for thin-genus $g(K)$ splittings.

Now suppose that $F$ is compressible. By a slight generalization of Lemma 1.1 of \cite{cassgord} this implies that $D_1$ and $D_2$ may be ambient isotoped to be disjoint. 

We now perform an operation known as `untelescoping' to the circular Heegaard surface $(F,S)$. See \cite{lacksimple,untel}. We replace the surface $S$ with three new surfaces, namely $S_{D_1}$, $S_{D_2}$ and $S_{D_1, D_2}$, where $S_{D_1}$ (resp. $S_{D_2}$) is the surface obtained by compressing $S$ along $D_1$ (resp. $D_2$) and $S_{D_1,D_2}$ is the surface obtained from $S$ by compressing along both $D_1$ and $D_2$. This process is illustrated in Figure \ref{simpuntel}.

\begin{figure}[h!]  
 \psfrag{H}[cc][Bl]{$D_2$}
  \psfrag{X}[cc][Bl]{$S_{D_1,D_2}$}
 \psfrag{C}[cc][Bl]{$C_1$}
 \psfrag{F}[Bc][Bl]{$F$}
 \psfrag{D}[cc][Bl]{$C_2$}
  \psfrag{J}[cc][Bl]{$K'$}
    \psfrag{W}[cc][Bl]{$S_{D_2}$}
 \psfrag{S}[Bc][Bl]{$S$}
  \psfrag{V}[Bl][Bl]{$S'_{t_1+\varepsilon}$}
  \psfrag{Q}[cc][Bl]{$S_{D_1}$}
  \psfrag{P}[Bl][Bl]{$=F_{\textrm{untwist}}$}
 \psfrag{G}[cc][Bl]{$D_1$}
 \psfrag{U}[cc][Bl]{$\textrm{Untelescope}$}
    \psfrag{L}[Bc][Bl]{$S=S_\pi$}
\centering
\includegraphics[width=0.8\textwidth]{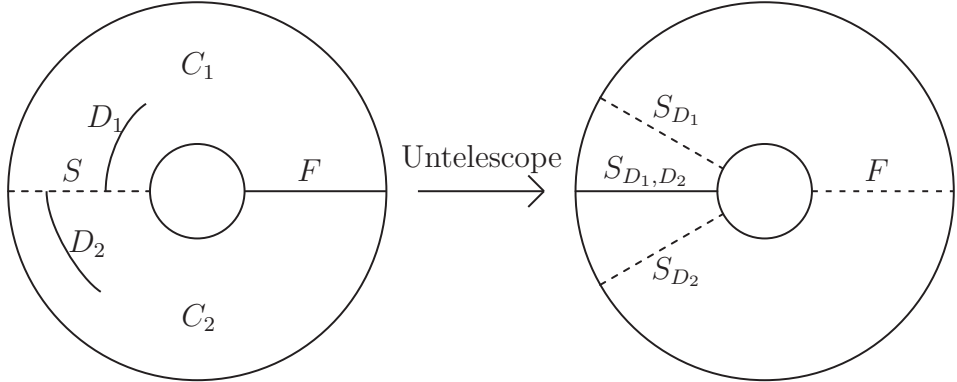}
\caption{Untelescoping a handle number one circular Heegaard splitting} \label{simpuntel}
\end{figure} 

Now, $S_{D_1}$ and $S_{D_2}$ are both parallel to $F$. Thus we may collapse the entire product region between $S_{D_1}$ and $S_{D_2}$ to a single surface $F$. Thus $(S_{D_1,D_2}, F)$ forms a new handle number one circular Heegaard splitting for $K$, illustrated in Figure \ref{simpuntel2}.  

\begin{figure}[h!]  
 \psfrag{H}[cc][Bl]{$D_2$}
  \psfrag{X}[cc][Bl]{$S_{D_1,D_2}$}
 \psfrag{C}[cc][Bl]{$C_1$}
 \psfrag{F}[Bc][Bl]{$F = S_{D_1}$}
 \psfrag{D}[cc][Bl]{$C_2$}
  \psfrag{J}[cc][Bl]{$K'$}
    \psfrag{W}[cc][Bl]{$S_{D_2}$}
 \psfrag{S}[Bc][Bl]{$S$}
  \psfrag{V}[Bl][Bl]{$S'_{t_1+\varepsilon}$}
  \psfrag{Q}[cc][Bl]{$S_{D_1}$}
  \psfrag{P}[Bl][Bl]{$=F_{\textrm{untwist}}$}
 \psfrag{G}[cc][Bl]{$=S_{D_2}$}
 \psfrag{U}[cc][Bl]{$\textrm{Untelescope}$}
    \psfrag{L}[Bc][Bl]{$S=S_\pi$}
\centering
\includegraphics[width=0.3\textwidth]{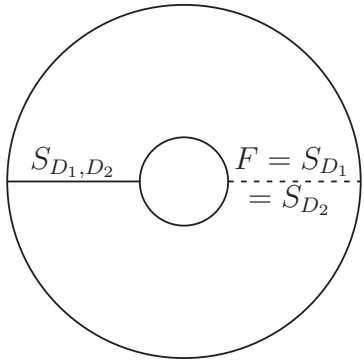}
\caption{The new handle number one circular Heegaard splitting $(S_{D_1,D_2}, F)$} \label{simpuntel2}
\end{figure} 

Notice how the role of $F$ changes under untelescoping: In $(F,S)$ it was the thin surface whereas in $(S_{D_1,D_2}, F)$ it is the thick surface. 

We now repeat this procedure. It must terminate because the thin-genus decreases by one each time, so eventually we arrive at a handle number one circular Heegaard splitting with incompressible thin surface. We have already shown that in this eventuality the thin surface must be a fiber and the thick surface must be a stabilized fiber. We may now recreate the circular Heegaard splitting $(F,S)$ we started with by performing the reverse procedure to untelescoping, namely amalgamation. This may be performed in an algorithmic fashion, using the same techniques as described by Lackenby in \cite{lacksimple}, until we reach a handle number one circular Heegaard splitting of the desired thin-genus. This concludes the proof of Theorem \ref{example}.\end{proof}

\bibliography{untwist}
\bibliographystyle{amsplain}

\end{document}